\documentclass[12pt]{amsart}
\usepackage{amssymb,latexsym, amsmath, amscd, array, graphicx}

\swapnumbers \numberwithin{equation}{section}

\theoremstyle{plain}

\newtheorem{thm}{Theorem}[section]

\newtheorem{prop}[thm]{Proposition}
\newtheorem{cor}[thm]{Corollary}

\theoremstyle{definition}
\newtheorem{defin}[thm]{Definition}

 \newcommand{\Wi}{\widetilde}

\def\scr{\mathcal}

\def\Z{{\mathbb Z}}
\def\Q{{\mathbb Q}}
\def\R{{\mathbb R}}
\def\N{{\mathbb N}}

\def\1{\hbox{\rm\rlap {1}\hskip.03in{\rom I}}}
\def\Bbbone{{\rm1\mathchoice{\kern-0.25em}{\kern-0.25em}
{\kern-0.2em}{\kern-0.2em}I}}

\long\def\forget#1\forgotten{} %

\newcommand\ver[1]{\marginpar{\tiny Changed in Ver \VER}}

\date{\today}
\begin{document}

\title[Macroscopic dimension and duality
groups]{Macroscopic dimension and duality groups}

\author[A.~Dranishnikov]{Alexander  Dranishnikov} %
\thanks{Supported by NSF grant DMS-0904278}

\address{Alexander N. Dranishnikov, Department of Mathematics, University
of Florida, 358 Little Hall, Gainesville, FL 32611-8105, USA}
\email{dranish@math.ufl.edu}



\begin{abstract}
We show that for a rationally inessential orientable closed
$n$-manifold $M$ whose fundamental group $\pi$ is a duality group
the macroscopic dimension of its universal cover
is strictly less than $n$:$$ \dim_{MC}\Wi M<n.$$ As a corollary we
obtain the following 
\begin{thm} The inequality $ \dim_{MC}\Wi M<n$ holds for the universal cover
of a closed spin $n$-manifold $M$ with a positive scalar curvature
metric if the fundamental group $\pi_1(M)$ is a virtual duality group virtually satisfying the
Analytic Novikov Conjecture.
\end{thm}
\end{abstract}

\maketitle

\section {Introduction}

In his book dedicated to Gelfand's 80th anniversary~\cite{Gr} Gromov
introduced the notion of macroscopic dimension and proposed a
conjecture that the macroscopic dimension of the
universal cover $\Wi M$ of a closed $n$-manifold with a positive
scalar curvature metric is at most $n-2$. This conjecture was proven
by D. Bolotov for 3-manifolds~\cite{B}. It was proved recently by
Bolotov and myself~\cite{BD} for spin $n$-manifolds, $n>3$ whose
fundamental group satisfies the Analytic Novikov Conjecture and the
Rosenberg-Stoltz K-theoretic condition~\cite{RS}: $ko_*(\pi)\to KO_*(\pi)$ is a
monomorphism. In this paper we show that in the case when $\pi$ is a
duality group the inequality $ \dim_{MC}\Wi M<n$ can be proven
without that $K$-theoretic condition.

We prove the inequality $\dim_{MC}\Wi M<n$ for all rationally
inessential orientable closed $n$-manifold $M$ whose fundamental
group is a duality group. The rational inessentiality means that
$f_*([M])=0$ in $H_n(B\pi;\Q)$. In the case of a spin manifold with
a positive scalar curvature metric the rational inessentiality
follows from Rosenberg's theorem and the KO-homology Chern-Dold
character.

\section{Homological obstruction to the inequality $\dim_{MC}\Wi M<n$}

A map $f:X\to Y$ of a metric space is {\em uniformly cobounded} if
there is a constant $C>0$ such that $diam(f^{-1}(y))<C$ for all
$y\in Y$.

We modified in~\cite{Dr}
Gromov's definition of macroscopic dimension as follows.
\begin{defin}
A metric space $X$ has the {\em macroscopic dimension} less than or
equal to $k$, $\dim_{MC}X\le k$, if there is a Lipschitz uniformly
cobounded map $f:X\to N^k$ to a $k$-dimensional simplicial complex.
\end{defin}
Here we assume that a simplicial complex has a uniform metric, i.e.
metric induced from the Hilbert space for the natural imbedding
$K\subset\Delta\subset\ell_2$ into the standard simplex in $\ell_2$.

We recall that locally finite homology groups $H^{lf}_*(\Wi X;G)$ of a complex $\Wi X$
with coefficients in an abelian group $G$ are defined by means of the chain complex
$C^{lf}_*(\Wi X;G)$
of infinite locally finite chains with coefficients in $G$. Thus,
$$C^{lf}_*(\Wi X;G)=\{\sum_{e\in E_n(\Wi X)} \lambda_ee\mid
\ \lambda_e\in G\}.$$ Here $E_n(\Wi X)$ denotes the set of $n$-cells
of a CW complex $\Wi X$. The local finiteness condition on a chain
requires that for every $x\in\Wi X$ there is a neighborhood such
that the number of $n$-cells $e$ intersecting $U$ for which
$\lambda_e\ne 0$ is finite. This condition is satisfied
automatically when $\Wi X$ is a locally finite complex.

Now let $\Wi X$ be the universal cover of a complex $X$ with fundamental group $\pi$.
The equivariant locally
finite homology groups $H_*^{lf,\pi}(\Wi X;G)$ with coefficients in an abelian group $G$
are defined by the subcomplex of infinite {\em
locally finite invariant chains}
$$ C_n^{lf,\pi}(\Wi X;G)=\{\sum_{e\in E_n(\Wi X)} \lambda_ee\mid
\lambda_{ge}=\lambda_e,\ \lambda_e\in G\}\subset C^{lf}_n(\Wi X;G).$$
We note that there is an obvious
equality $H_*(X;G)=H_*^{lf,\pi}(\Wi X;G)$.

A normed abelian group $(G,|\ |)$ consists of an abelian group $G$
and a nonnegative function $|\ |:G\to\R$ that satisfies the axioms
(1) $|a|=|-a|$; (2) $|a+b|\le|a|+|b|$; (3) $|a|=0$ iff $a=0$. We
consider the natural norm on $\Z$ and the norm with value 1 on
$\Z_p\setminus\{0\}$.

Let $G$ be a normed abelian group. A chain $$\sum_{e\in E_n(\Wi X)}
\lambda_ee\in C_n^{lf}(\Wi X;G)$$ is called {\em bounded} if there
is a uniform bound on $\lambda_e$. Note that bounded infinite chains
form a chain subcomplex  $C_*^{(\infty)}(\Wi X;G)\subset
C^{lf}_*(\Wi X;G)$. The homology groups $H^{(\infty)}_n(\Wi X;G)$ of
that complex are called {\em $\ell_{\infty}$-homology} in view  of the analogy with the
$\ell_{(\infty)}$-cohomology defined in~\cite{Ge}. Note that these homology groups first appeared 
in~\cite{BW} under the name {\em uniformly finite homology} .

The
inclusion of chain groups $C_*^{lf,\pi}(\Wi X;G)\to
C_*^{(\infty)}(\Wi X;G)$ defines a homomorphism
$$ pert_*^X :H_*(X;G)\to H_*^{(\infty)}(\Wi X;G).$$

In the case of a discrete group $\pi$
we will use notations $$H_*(\pi;G)=H_*(B\pi;G),
\ \ \ \ H^{(\infty)}(\pi;G)= H_*^{(\infty)}(E\pi;G),$$
$$H^{lf}_*(\pi;G)=H^{lf}_*(E\pi;G),\ \ \text{and}\ \
pert^{\pi}_*=pert^{B\pi}_*.$$

Note that $\ell_{\infty}$-homology is a special case (when
coefficient group $G$ is a trivial $\pi$-module) of the almost
equivariant homology  defined in \cite{Dr}. Thus, Theorem 4.2
of~\cite{Dr} in the language of $\ell_{\infty}$-homology can be
rephrased as follows:
\begin{thm} \label{Obstruction}
For a closed orientable $n$-manifold $M$ the following conditions are equivalent:

1. $\dim_{MC}\Wi M<n$;

2. $f_*([M])\in \ker\{pert_*^{\pi}:H_n(\pi;\Z)\to
H_n^{(\infty)}(\pi;\Z)\}$ where $f:M\to B\pi$ is a map classifying
the universal cover of $M$.

\end{thm}

\section{Main Result}

We recall~\cite{Br} that a group $\pi$ of type FP
is called {\em a duality group} if for all $i$ there are natural isomorphisms
$$H^i(\pi,-)\to H_{n-i}(\pi, D\otimes -)$$
where $n=cd(\pi)$, the cohomological dimension of $\pi$, and $D=H^n(\pi,\Z\pi)$.
If $D=\Z$ then $\pi$ is called  a {\em Poincare duality group}.
Examples of Poincare duality groups are the fundamental groups of closed
aspherical manifolds. Examples of duality groups different from Poincare
duality groups include free groups, all knot groups, torsion free arithmetic
groups and the
products of all the above.

\begin{prop}\label{finite coefficients}
For any duality group $\pi$ with $cd(\pi)=N$, and $p\in\N$,
$$H^{(\infty)}_n(\pi;\Z_p)=0 \ \ \text{for}\ \  n\ne N$$ and
$$H^{(\infty)}_N(\pi;\Z_p)=\oplus\Z_p\ne 0.$$
\end{prop}
\begin{proof} It suffices to consider the case when $p$ is prime.
Since the coefficient group is finite, $H^{(\infty)}_n(\pi;\Z_p)=H^{lf}_n(\pi;\Z_p).$
Note that $$H^{lf}_n(\pi;\Z_p)=H^{lf}_n(E\pi;\Z_p)=H^n_c(E\pi;\Z_p)=
H^n(\pi;\Z\pi\otimes\Z_p)=0$$
for all $n\ne N$. Here the second equality is the duality isomorphism between
homology and cohomology
with coefficient in a field applied to the one-point compactification of $E\pi$.
 For the last two equalities see ~\cite{Br}, Chapter VIII, Lemma 7.4 
and Theorem 10.1 respectively.

By the cited Theorem 10.1, $H^N(\pi,\Z\pi)$ is torsion free as an abelian group.
Therefore, $H^N(\pi,\Z\pi)\otimes\Z_p=\oplus\Z_p$.
The Universal Coefficient Formula for the cohomology with compact supports and the
equality $H^i_c(E\pi;\Z)=H^i(\pi,\Z\pi)$ imply that $H^N_c(E\pi;\Z_p)=\oplus\Z_p$.
Hence, $H^{(\infty)}_N(\pi;\Z_p)=\oplus\Z_p.$
\end{proof}

\begin{cor}\label{Cor1}
For every duality group $\pi$ with $cd(\pi)=N$ and any $n\ne N-1$,
the homology group $H^{(\infty)}_n(\pi)$ does not contain elements of finite order.
\end{cor}
\begin{proof}
In view of the Universal Coefficient Formula
$$
0\to H_n^{(\infty)}(\pi)\otimes\Z_p \to H_n^{(\infty)}(\pi;\Z_p) \to
Tor(H_{n-1}^{(\infty)}(\pi),\Z_p)\to 0
$$
Proposition~\ref{finite coefficients} implies that
$H_n^{(\infty)}(\pi)$ does not have  $p$-torsions for $n\ne N-1$.
\end{proof}
REMARK. One can derive that the $p$-torsion subgroup of $H_N^{(\infty)}(\pi)$ is $\oplus\Z_{p^{\infty}}$
where $\Z_{p^{\infty}}=\lim_{\rightarrow}\Z_{p^k}$.
Most likely for duality group this sum is empty.
Corollary~\ref{Cor1} and Theorem~\ref{Obstruction} imply the
following.
\begin{cor}\label{Cor2}
Suppose that an orientable closed $n$-manifold $M$ with the duality fundamental
group $\pi$ is rationally
inessential and $cd(\pi)\ne n+1$. Then $\dim_{MC}\Wi M< n$.
\end{cor}
We recall that a group $\pi'$  virtually satisfies some property $\scr P$ if
it contains a finite index
torsion free subgroup $\pi$ with that property. In particular,
a group $\pi'$ is said to be a {\em virtually duality group} if some
subgroup $\pi\subset\pi'$ of finite
index is a duality group~\cite{Br} and a group $\pi'$
virtually satisfies the Analytic Novikov Conjecture if it contains a finite index subgroup $\pi$
that satisfies the Analytic Novikov Conjecture.

\begin{thm}
Suppose that the fundamental group $\pi'$ of a spin manifold $M'$
with the positive scalar curvature is a virtual duality group which
virtually satisfies the Analytic Novikov Conjecture. Then
$$\dim_{MC}\Wi M'< n.$$
\end{thm}
\begin{proof} The conditions imply that there is a subgroup $\pi\subset\pi'$ of
finite index which is a duality group and satisfies the  Analytic
Novikov Conjecture. Let $p:M\to M'$ be a finite covering that
corresponds to the subgroup $\pi$. Clearly, $M$ is spin and  a
metric of positive scalar curvature on $M'$ lifts to a metric with
positive scalar curvature on $M$. We show that  $\dim_{MC}\Wi M< n$.
The required result would follow from the fact that the manifolds
$M$ and $M'$ have the same universal cover $\Wi M=\Wi M'$.

Let $f:M\to B\pi$ be a classifying map of the universal cover $\Wi
M$. By Rosenberg's theorem~\cite{R} we have $\alpha\circ
f_*([M]_{KO})=0$ where
$$\alpha: KO_*(B\pi)\to KO_*(C_r(\pi))$$ is the analytic assembly map
and $[M]_{KO}$ is the $KO$-fundamental class defined by the spin
structure. The Analytic Novikov Conjecture for $\pi$ says that
$\alpha$ is a monomorphism. Therefore, $f_*([M]_{KO})=0$. There is a
natural isomorphism called the Chern-Dold character
$$KO_n(X)\otimes\Q\to\bigoplus_{i\in\Z} H_{n+4i}(X;\Q)$$ (see for
example~\cite{Ru}, Theorem-Definition 7.13). We note that under the
Chern-Dold character isomorphism the rationalized KO-fundamental
class $[M]_{KO}$ of $M$ is taken to an element
$a=(a_i)_{i\in\Z}\in\oplus_{i\in\Z} H_{n+4i}(M;\Q)$ with $a_0\ne 0$.
This is an obvious fact for $M=S^n$. For general $M$ it can be shown
by taking a degree one map of $M$ onto the $n$-sphere $S^n$.
Therefore, $f_*([M])=0$ in  $H_n(B\pi;\Q)$. Then
Corollary~\ref{Cor2} implies the inequality  $\dim_{MC}\Wi M< n$.
This completes the proof if $cd(\pi)\ne n+1$.

The case $cd(\pi)=n+1$ is treated separately. We claim that in this case $M$
is integrally inessential.
 Since the complex $B\pi/B\pi^{(n-1)}$ is $(n+1)$-dimensional and
 $(n-1)$-connected, it is homotopy equivalent to the wedge of spheres
 and Moore spaces $\vee_i S^n\bigvee\vee_j M(\Z_{m_j},n)\bigvee\vee_kS^{n+1}$.
 We note that any essential map of $S^n$ to $S^n$ or to
 the Moore space $M(\Z_m,n)$ takes the KO-fundamental class $[S^n]_{KO}$ to
 nonzero. We may
assume that $M$ has a CW structure with one top dimensional cell and
the map $f$ is cellular. The commutativity of the diagram
$$
\begin{CD}
M_0 @>f>> B\pi\\
@VpVV @VqVV\\
S^n=M_0/M^{(n-1)}_0 @>\bar f>> B\pi/B\pi^{(n-1)}\\
\end{CD}
$$
and the fact that $f_*([M]_{KO})=0$ imply that $\bar
f_*([S^n]_{KO})=0$. This implies that $\bar f$ is null-homotopic and
hence it induces a zero homomorphism of the integral homology
groups. Since the quotient map $q:B\pi\to B\pi/B\pi^{(n-1)}$ induces
a monomorphism of $n$-dimensional homology groups, $f_*([M])=0$ in
$H_n(B\pi;\Z)$.

\end{proof}

\end{document}